\documentclass[12pt]{amsart}
\usepackage[utf8]{inputenc}
\usepackage{verbatim}
\usepackage{amsmath, amsthm,amssymb,bbm,enumerate,mathrsfs,mathtools}

\title{Towards Nash-Williams Orientation Conjecture for Infinite Graphs}
\author[A.~Assem]{Amena Assem}
\address{Department of Combinatorics and Optimization, University of Waterloo, ON, Canada}
\email{a36mahmo@uwaterloo.ca}

\date{}

\keywords{Nash-Williams, infinite graph, orientation, edge-connectivity, lifting}

\usepackage{xcolor} 	
\usepackage[unicode]{hyperref}
\hypersetup{
	colorlinks,
    linkcolor={red!60!black},
    citecolor={green!60!black},
    urlcolor={blue!60!black},
}
\usepackage[abbrev, msc-links]{amsrefs}

\usepackage{geometry}
\theoremstyle{plain}
\newtheorem{theorem}{Theorem}[section]

\newtheorem{lemma}[theorem]{Lemma}

\newtheorem{conjecture}[theorem]{Conjecture}

\theoremstyle{definition}

\newtheorem{definition}[theorem]{Definition}

\newtheorem{notation}[theorem]{Notation}
\newtheorem{case}{Case}

\begin{document}

\maketitle

\begin{abstract}
In 1960 Nash-Williams proved that an edge-connectivity of $2k$ is sufficient for a finite graph to have a $k$-arc-connected orientation. He then conjectured that the same is true for infinite graphs. In 2016, Thomassen, using his own results on the auxiliary \emph{lifting graph}, proved that $8k$-edge-connected infinite graphs admit a $k$-arc connected orientation. Here we improve this result for the class of $1$-ended locally-finite graphs and show that an edge-connectivity of $4k$ is enough in that case. Crucial to this improvement are results presented in a separate paper, by the same author of this paper, on the key concept of the lifting graph, extending results by Ok, Richter, and Thomassen.
\end{abstract}

\section{Introduction}

An orientation of a graph $G$ is \emph{$k$-arc-connected} if for any two vertices $x$ and $y$ in $G$ there are $k$ arc-disjoint directed paths from $x$ to $y$ in the oriented graph. It is clear that an edge-connectivity of $2k$ is necessary for the existence of such an orientation. Nash-Williams proved in 1960 that the condition of being $2k$-edge-connected is also sufficient for finite graphs to admit a $k$-arc-connected orientation \cite{nash1960orientations}. He claimed that the same is true for infinite graphs, but still, after more than $60$ years, the question remains open. Until 2016, it was not even known whether a function $f(k)$ of positive integers exists such that every $f(k)$-edge-connected infinite graph has a $k$-arc-connected orientation. Then Thomassen proved his breakthrough result that an edge-connectivity of $8k$ is enough \cite{thomassen2016orientations}.

\begin{conjecture}\label{conjecture of Nash Williams} (Nash-Williams \cite{nash1960orientations})
Let $k$ be a positive integer. Every $2k$-edge-connected infinite graph admits a $k$-arc-connected orientation.
\end{conjecture}

\begin{theorem} (Thomassen \cite{thomassen2016orientations})
Let $k$ be a positive integer, and let $G$ be an $8k$-edge-connected infinite graph. Then $G$ admits a $k$-arc-connected orientation.
\end{theorem}

To prove his theorem, Thomassen first, in the same paper, proved some results on the lifting graph, which is an auxiliary graph used in edge-connectivity proofs. Lifting, also sometimes called splitting, was studied first by Lov\'asz and Mader, and later by Frank (\cite{lovasz1976eulerian}, \cite{mader1978reduction}, \cite{frank}). More results on lifting were proved in 2016 by Ok, Richter, and Thomassen \cite{ORT2016linkages} and they used it in their proof of a theorem on edge-disjoint linkage in infinite graphs. 

Lifting at a vertex $s$ is the operation of deleting two edges incident with $s$ and adding an edge between their non-$s$ end-vertices. Careful lifting, done such that the local edge-connectivity is somewhat preserved, is what is needed. Lifting is important to proving the existence of a $k$-arc-connected orientation both in finite and infinite graphs. This can be seen in Mader's paper \cite{mader1978reduction} where he used his result on lifting to give a simpler proof of the orientation theorem of  Nash-Williams for finite graphs. 

In 2022, the author of this paper made a comprehensive study of the lifting graph in \cite{Assem2022lifting}. She will use her results to prove in this paper that an edge-connectivity of $4k$ is enough to have a $k$-arc-connected orientation in $1$-ended locally-finite graphs. The author is thankful for the several discussions and meetings with Bruce Richter trying to extend the result for graphs with multiple ends.

In our proof, we follow the same general lines of Thomassen's proof. Thomassen first proved that for a finite set of vertices $A$ in a $4k$-edge-connected locally-finite graph $G$, there is an immersion in $G$ of a finite Eulerian $2k$-edge-connected graph with vertex set $A$ (Theorem 4 in \cite{thomassen2016orientations}). We will prove a similar result in Section 3 on immersions. The property of being Eulerian (every vertex has even degree) with the even connectivity was needed to be able to use the fact that the lifting graph has a disconnected complement in that case (\cite{thomassen2016orientations},\cite{ORT2016linkages}). Now we know from \cite{Assem2022lifting} more about the lifting graph when the degree of the vertex at which the lifting takes place is odd, so Eulerian graphs will not show up in our work, and we will only need a smaller edge-connectivity. However, we will need the condition of being $1$-ended.

Thomassen then generalized the result from locally-finite graphs to countable graphs. He proved Theorem 9 in \cite{thomassen2016orientations} about an operation called the splitting (of vertices) of an infinite countable graph such that one gets a graph of the same edge-connectivity with locally-finite blocks. After that he generalized, using Zorn's Lemma, to uncountable graphs. The method of generalization is explained in Section 8 of Thomassen's paper \cite{thomassen2016orientations}. One difficulty with this generalization in our case is that our result here is for $1$-ended graphs, and it is not guaranteed that the property of being $1$-ended is preserved after splitting vertices.

The result of this paper is further improved in \cite{Max2023orientation} in collaboration with Koloschin and Pitz, as we showed that the conjecture is true for locally-finite graphs with countably many ends.

\section{Lifting}

In this section we include the definitions and theorems we need on lifting. As mentioned in the introduction, the lifting graph is crucial in our proofs. Graphs in this paper are loopless multigraphs. If $s$ is a vertex in a finite graph $G$ and $sv$ and $sw$ are two edges incident with $s$, then to \textit{lift} the pair of edges $sv$ and $sw$ is to delete them and add the edge $vw$ if $v\neq w$ (but nothing is added if the two edges are parallel). 

The edges $sv$ and $sw$ are $k$-\textit{liftable}, or they form a $k$-\textit{liftable} pair, if after lifting them there are, in the resulting graph, $k$ pairwise edge-disjoint paths between any two vertices different from $s$. We simply say \textit{liftable} when connectivity is understood from the context. We also say that the edge $sv$ \textit{lifts} with $sw$. We say that $G$ is $(s,k)$-\emph{edge-connected} if between any two vertices different from $s$ in $G$ there are $k$ edge-disjoint paths (that may possibly go through $s$).

\begin{definition}(\cite{thomassen2016orientations}, \cite{ORT2016linkages}, \cite{Assem2022lifting})
Let $k$ be a positive integer and $s$ a vertex in a finite graph $G$ such that $G$ is $(s,k)$-edge-connected. The $k$-\textit{lifting graph} $L(G,s,k)$ is the finite graph whose vertex set is the set of the edges of $G$ incident with $s$ and in which two vertices are adjacent if they form a $k$-liftable pair of edges in $G$.
\end{definition}

\begin{notation}
For a set $X$ of vertices in a graph $G=(V,E)$, we write $\delta(X)$ to denote the set of edges with one end in $X$ and one end in $V\setminus X$.
\end{notation}

\begin{figure}
  \centering
  \begin{minipage}[b]{0.45\textwidth}
    \includegraphics[width=\textwidth]{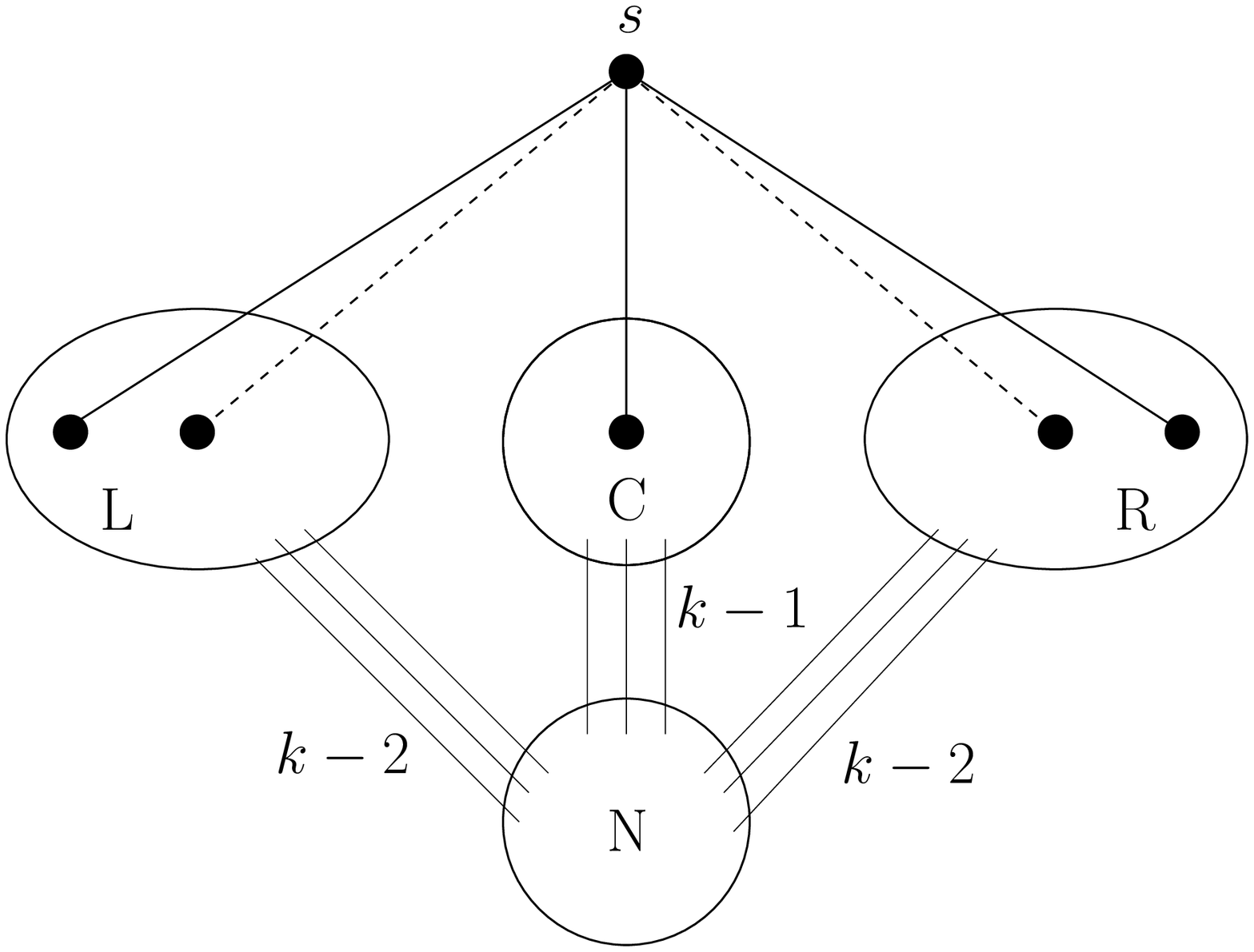}
  \end{minipage}
  
  \vspace{0.25in}
  \begin{minipage}[b]{0.45\textwidth}
    \includegraphics[width=\textwidth]{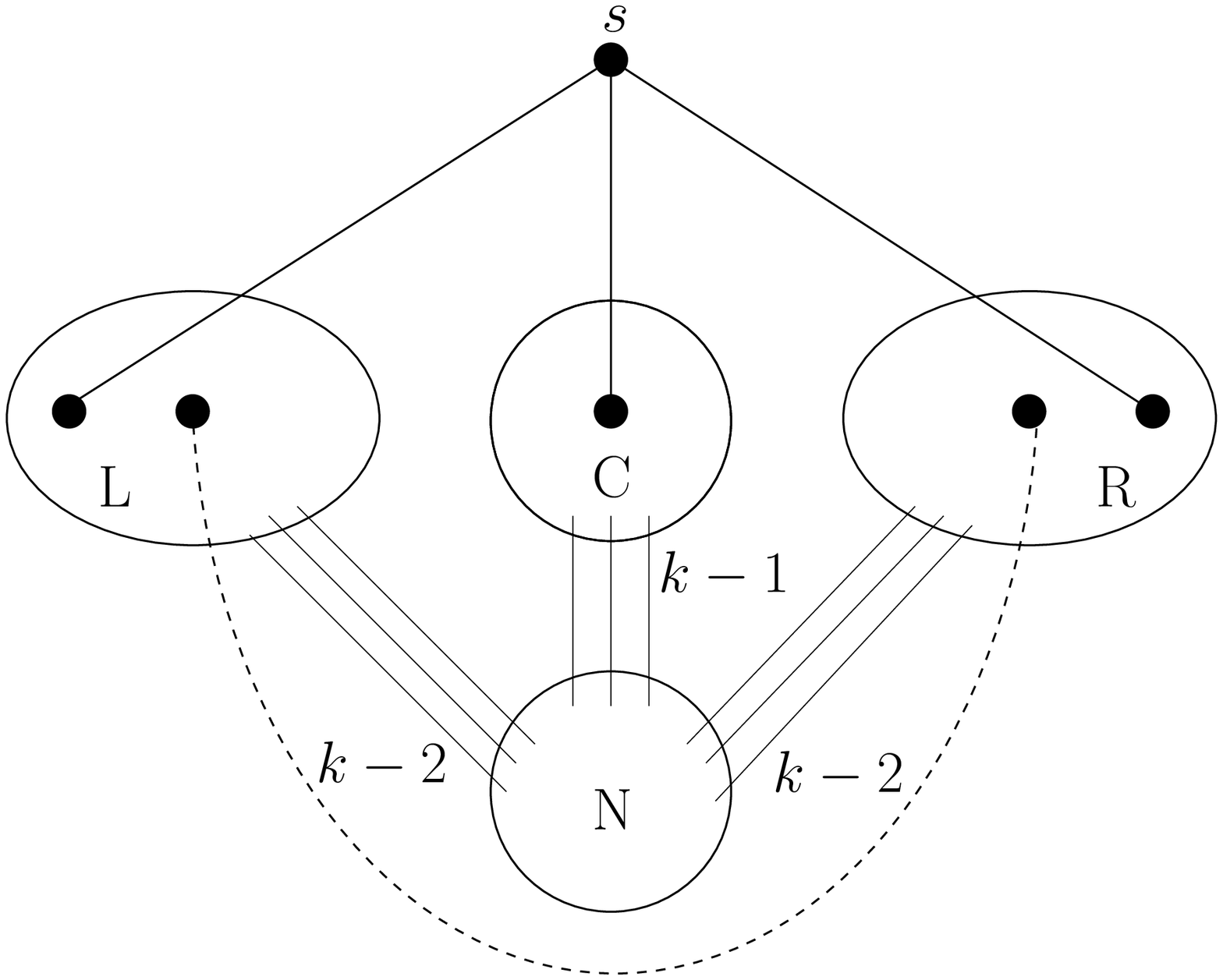}
    \end{minipage}
 \caption{Here $k$ may be even as well as odd. This is a graph with $deg(s)=5$ whose lifting graph $L(G,s,k)$ is the union of an isolated vertex (the middle edge incident with $s$ in the drawing) and $K_{2,2}$ (the two edges on the right and two edges on the left). If the dashed pair of edges is lifted, then any pair of the remaining three edges is not liftable since such an additional lift would result in the reduction of the size of one of the cuts $\delta(L\cup C \cup N)$ or $\delta(R\cup C \cup N)$ from $k+1$ to $k-1$.}
    \label{example 2 lifting from deg 5 to deg 3}
\end{figure}

As mentioned earlier, our proof of the main theorem relies on results concerning the lifting graph from \cite{Assem2022lifting}. In particular, combining point (ii) in \cite[Theorem~1.5]{Assem2022lifting} (for $\deg(s)>4$) and \cite[Proposition~3.4]{ORT2016linkages} (for $\deg(s)=4$) directly gives us the following theorem, which is also a special case of \cite[Theorem~3.3]{Max2023orientation} (when $A=V\setminus \{s\}$).

\begin{theorem}\label{main lifting theorem}
Let $G$ be an $(s,k)$-edge-connected graph such that $k\geq 2$ and $\deg(s)\geq 4$. If $k$ is even, then

\begin{itemize}

\item either the complement of $L(G,s,k)$ is disconnected, or 

\item $L(G,s,k)$ is the union of an isolated vertex and a balanced complete bipartite graph.

\end{itemize}
\end{theorem}

In 2016 it was already known by Thomassen \cite[Theorem~2]{thomassen2016orientations} that the complement of $L(G,s,k)$ is disconnected if $G$ is Eulerian. In the same year, Ok, Richter, and Thomassen \cite[Theorem~1.2]{ORT2016linkages} showed that if both $\deg(s)$ and $k$ are even, then $L(G,s,k)$ is complete multipartite (so its complement is disconnected). The latter result was previously proved by Jord\'an in 1999 and published in 2003 \cite[Theorem~3.2]{Jordan} for applications in connectivity augmentation before being proved independently by Ok, Richter, and Thomassen for infinite graph applications. 

Although Ok, Richter, and Thomassen recognized a possible structure of an isolated vertex plus a complete multipartite graph when $\deg(s)$ is odd, they knew it was bipartite only when $\deg(s)=5$, and could not deal with the general case when $\deg(s)$ is odd nor when the complement of $L(G,s,k)$ is connected in applications. This required them to assume a higher edge-connectivity in their theorems. If $k$ is even and $\deg(s)$ is odd, then it is possible that the complement of $L(G,s,k)$ is connected as well as disconnected. If it is connected, then it is two cliques of the same size intersecting in one vertex, in other words, $L(G,s,k)$ is an isolated vertex plus a balanced complete bipartite graph.

To give the reader a better sense of how lifting works, we discuss briefly some of the cases in which $\deg(s)$ is small. Note that when $\deg(s)=2$ then the two edges incident with $s$ are liftable because if any path between two vertices different from $s$ goes through one of these two edges then it has to go through the other. The case when $\deg(s)=3$ is problematic and Section 7 of \cite{Assem2022lifting} is dedicated for studying it. 

For an example, Figure \ref{example 2 lifting from deg 5 to deg 3} depicts a graph $G$ with $\deg(s)=5$ whose lifting graph is the union of an isolated vertex and a $K_{2,2}$, as will be explained, and also illustrates the graph resulting from lifting a pair of edges incident with $s$. We assume that each of the sets $R$ and $L$ are internally well-connected. Note first that the set of edges incident with $s$ whose other end-vertices are in $R\cup C$ are pairwise non-liftable, i.e. they form an independent set in $L(G,s,k)$. This is the case because the cut $\delta(R\cup C \cup N)$ has size $k+1$ and so any lift of a pair of edges with end-vertices in $R\cup C$ results in the reduction of the size of this cut to $k-1$. As can be seen in the figure, the number of edges between $s$ and $R\cup C$ is $(\deg(s)+1)/2$ which is the maximum possible size of an independent set in $L(G,s,k)$ by Frank's theorem \cite[Theorem~B]{frank}. The same can be said about the set of edges between $s$ and $L\cup C$. Thus, the middle edge in the figure (between $s$ and $C$) is the isolated vertex in $L(G,s,k)$. Any edge of the two on the right side is liftable with any edge of the two on the left side if we assume that each of $R$ and $L$ has connectivity at least $k/2$ to avoid the emergence of small cuts after lifting. However, lifting any such liftable pair of edges results in a graph in which $\deg(s)=3$, and in this new graph any pair of edges incident with $s$ is not liftable because lifting will result in a cut of size $k-1$, and this is an example of a difficult situation with $\deg(s)=3$ as remarked earlier.

\section{Sequence of lifts in an infinite graph}

Here we present a technical lemma which is a cornerstone for the proof of the main theorem. In this lemma lifting is used to find edge-disjoint paths between certain pairs of edges in an infinite graph. These paths will be used in connecting the branch vertices of an immersion (to be defined in Section 4). The general argument is similar to the arguments in the proofs of Theorem 4 in \cite{thomassen2016orientations} by Thomassen and Theorem 1.3 in \cite{ORT2016linkages} by Ok, Richter, and Thomassen. We inductively pair as many as possible of the edges of a cut with one finite side such that certain properties are satisfied that reconcile the connectivity on the finite side with the topology on the infinite side. Recall that a \textit{ray} is a one-way infinite path, and an \textit{end} is an equivalence class of rays where two rays are in the same end (class) if there are infinitely many vertex disjoint paths between them. An infinite graph is \emph{locally-finite} if the degree of every vertex is finite. For locally-finite graphs, being $1$-ended (having only one end) is equivalent to the property that the deletion of any finite set of vertices results in only one infinite component.

\begin{lemma}\label{lifting sequence}
Let $k$ be a positive even integer and $G$ a $k$-edge-connected locally-finite graph. Suppose that there exists a finite set $S$ of vertices such that every edge of $\delta(S)$ is the first edge of a ray in an edge-disjoint finite collection of rays $\mathcal{R}$ in one end of $G$. Let $s$ be the contraction vertex in $G/(G-S)$.
If $\deg(s)$ is even, let $I:=\{0,\ldots, \deg(s)/2\}$, and if it is odd, let $I:=\{0, \ldots, (\deg(s)-3)/2\}$. Then there is a sequence $\{G_i\}_{i\in I}$ of finite graphs such that $G_0=G/(G-S)$, and for every $i\in I$, $i>0$, $G_i$ is obtained from $G_{i-1}$ by lifting a pair of edges $e_i$ and $e'_i$ incident with $s$ such that (see Figure \ref{ORT-proof}):

\begin{itemize}
\item [(1)] the pair $\{e_i,e_i'\}$ is $k$-liftable in $G_{i-1}$, and
\item [(2)] there is a path $P_i$ in $G-S$ joining the end-vertices of $e_i$ and $e_i'$ in $G-S$ such that $P_i$ is edge-disjoint from $\bigcup_{j<i} P_j$ and from all the rays in $\mathcal{R}$ that begin with edges in $\delta(S)\setminus\{e_j, e_j': j\leq i\}$.
\end{itemize}

Moreover, in case $\deg(s)$ is odd, if for some $i\in I$ the lifting graph $L(G_i,s,k)$ is the union of an isolated vertex and a balanced complete bipartite graph, then the lifting sequence and the paths $\{P_j\}_{j\in I}$ can be chosen so that for the three edges $e,e^*,e'$ of $\delta(S)$ that remain at the end of the sequence, there is a vertex $w\in G-S$ and edge-disjoint paths $W, W^*, W'$ from $w$ to respectively $e,e^*,e'$ that are also edge-disjoint from $\{P_j\}_{j\in I}$.
\end{lemma}

\begin{figure}
\centering
   \includegraphics[scale=0.7]{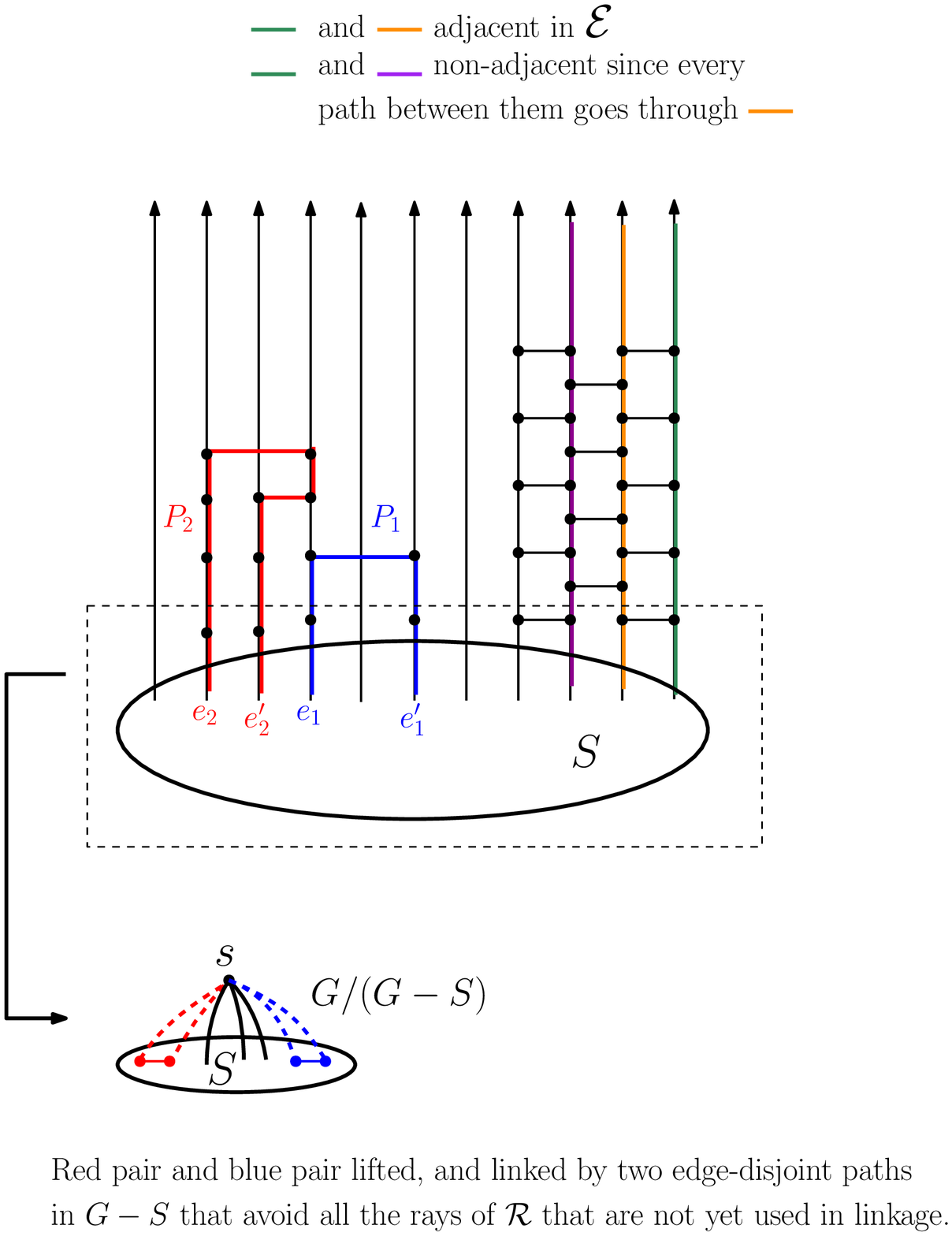} 
   \caption{Pairing the edges of $\delta(S)$ and linking them by edge-disjoint paths on the infinite side of the cut.}
    \label{ORT-proof}
\end{figure}

\begin{proof}
The proof follows almost the same general lines as in the proofs of Theorem 4 in \cite{thomassen2016orientations} by Thomassen and Theorem 1.3 in \cite{ORT2016linkages} by Ok, Richter, and Thomassen. The significantly new part here is dealing with the case when the lifting graph is the union of an isolated vertex and a balanced complete bipartite graph, and finding the vertex $w$ and the paths $W, W^*, W'$.

Note first that $\deg(s)\geq k \geq 2$. If $\deg(s)=2$, then the unique pair of edges is liftable and we can take any path to connect them in $G-S$. If $\deg(s)=3$, we do not lift anything but we will show later how to find the paths $W, W^*, W'$ as part of the general case for odd degree. So from here on, we may assume that $\deg(s)\geq 4$.

We inductively define the pairs of edges as follows. Let $i$ be at least $1$ and assume that $\{e_j,e_j'\}$, $G_j$, and $P_j$ are defined for all $j<i$. Let $\delta_i(S)$ denote the set $\delta(S)\setminus\{e_j, e_j' : j<i\}$ and $\mathcal{R}_i$ denote the set of rays in $\mathcal{R}$ that begin with edges in $\delta_i(S)$.

 We consider two graphs on the same vertex set $\delta_i(S)$. The first graph is the lifting graph $\mathcal{L}_i = L(G_{i-1},s,k)$, which encodes information about the edge-connectivity on the finite side of the cut $\delta(S)$. The second graph, to be defined below, encodes information about the topology on the infinite side of the cut.

The second graph on $\delta_i(S)$ is the \textit{end graph} $\mathcal{E}_i$. Distinct edges $e$ and $e'$ from $\delta_i(S)$ are adjacent in $\mathcal{E}_i$ if there are infinitely many vertex-disjoint paths in $G-S$ that:
       
       \begin{itemize}
           \item join the two rays in $\mathcal{R}_i$ beginning with $e$ and $e'$, and
           \item are edge-disjoint from all the other rays in $\mathcal{R}_i$.
       \end{itemize}

Note that the end graph $\mathcal{E}_i$ is connected because the rays of $\mathcal{R}$ are in one end and are only finitely many.

If the complement of $\mathcal{L}_i$ is disconnected, then since $\mathcal{E}_i$ is connected, there is a pair of elements $e_i$ and $e_i'$ from $\delta_i(S)$ that are adjacent in $\mathcal{E}_i$ but not adjacent in the complement of $\mathcal{L}_i$, and so adjacent in $\mathcal{L}_i$. This means that $e_i$ and $e_i'$ are liftable in $G_{i-1}$ and also there are infinitely many vertex disjoint paths between the two rays in $\mathcal{R}_i$ beginning with them and these paths are edge-disjoint from all the other rays in $\mathcal{R}_i$. 

Now assume that the complement of $\mathcal{L}_i$ is connected. Then by Theorem \ref{main lifting theorem}, since $k$ is even and $\deg(s)\geq 4$, it follows that $\mathcal{L}_i$ is the union of an isolated vertex and a balanced complete bipartite graph. Let $e^*$ denote the unique edge in $\delta_i(S)$ that is isolated in $\mathcal{L}_i$, and $R^*$ the unique ray in $\mathcal{R}$ that begins with it. We have the following two cases.

\begin{case}
    There are two edges $e_i$ and $e'_i$ from $\delta_i(S)$ that are on different sides of the complete balanced bipartite graph $\mathcal{L}_i-e^*$ and are adjacent in $\mathcal{E}_i$.
\end{case}

In this case, $e_i$ and $e'_i$ is the pair of edges we lift. Having determined, in every case considered so far, how to choose $e_i$ and $e_i'$ satisfying $(1)$ and $(2)$, now choose $Q_i$ to be one of the infinitely many vertex-disjoint paths between the two rays $R_i$ and $R_i'$ in $\mathcal{R}_i$ containing $e_i$ and $e_i'$ such that $Q_i$ is edge-disjoint from the finitely many edges of $\bigcup_{j<i} P_j$ and from the rays in $\mathcal{R}_i\setminus \{R_i,R_i'\}$. The desired path $P_i$ is the path that begins with the end-vertex of $e_i$ in $G-S$, goes along $R_i$ until it meets $Q_i$, then along $Q_i$, then along $R_i'$ down to $e_i'$. The remaining case needs a special treatment in finding these linking paths.

\begin{case} 
     $e^*$ is a cut-vertex in $\mathcal{E}_i$ between the two sides of $\mathcal{L}_i-e^*$.
\end{case}

Let $e$ and $e'$ be neighbours of $e^*$ on different sides of $\mathcal{L}_i-e^*$ and let $R$ and $R'$ be the rays of $\mathcal{R}_i$ that begin with them. Such neighbours of $e^*$ exist because $\mathcal{E}_i$ is connected. By definition of adjacency in $\mathcal{E}_i$ there are infinitely many vertex-disjoint paths between $R$ and $R^*$ that are edge-disjoint from the other rays in $\mathcal{R}_i$ and infinitely many vertex-disjoint paths between $R^*$ and $R'$ that are edge-disjoint from the other rays in $\mathcal{R}_i$. Pick a path $P$ between $R$ and $R^*$ that is edge-disjoint from the other rays in $\mathcal{R}_i$ and a path $P'$ between $R^*$ and $R'$ that is edge-disjoint from the other rays in $\mathcal{R}_i$ such that these two paths are also edge-disjoint from $\bigcup_{j<i}P_j$. Let the end-vertices of $P$ and $P'$ on $R^*$ respectively be $w$ and $w'$. Assume without loss of generality that $w$ is closer to $S$ on $R^*$ than $w'$, and let $Q$ be the path contained in $R^*$ between $w'$ and $w$. Let $W^*$ be the path contained in the ray $R^*$ from $w$ to $e^*$, let $W$ be the path from $w$ along $P$ and then down on $R$ to $e$, and let $W'$ be the path from $w$ up along $Q$ to $w'$ then along $P'$ then down on $R'$ to $e'$. Then $W$, $W^*$, and $W'$ are edge-disjoint from the rays in $\mathcal{R}_i\setminus \{R,R^*,R'\}$ and from $\bigcup_{j<i}P_j$.

The pair of edges, $e$ and $e'$ will not be lifted. The edges $e$, $e^*$, and $e'$ will remain to the end of the sequence of lifts. However, the pair of edges to be lifted at this step will be chosen as follows. Assuming $\deg_{G_{i-1}}(s)>3$, let $\mathcal{E}'_i$ be the end graph on $\delta_i(S)\setminus \{e,e^*,e'\}$, where two edges are adjacent if between the two rays beginning with them in $\mathcal{R}_i\setminus \{R,R^*,R'\}$ there are infinitely many vertex disjoint paths that are edge-disjoint from the other rays in $\mathcal{R}_i\setminus \{R,R^*,R'\}$. Note that this means that we no longer care if the connections between two rays go through the rays $R,R^*,R'$, and they have to go through some of them since $e^*$ is a cut-vertex in $\mathcal{E}_i$ and $e$ and $e'$ are direct neighbours of it.

The graph $\mathcal{E}'_i$ is a connected graph, and so there is a pair of edges $e_i$ and $e_i'$ in $\delta_i(S)\setminus \{e,e^*,e'\}$ from two different sides of the bipartite graph $\mathcal{L}_i-\{e,e^*,e'\}$, that is adjacent in $\mathcal{E}'_i$.
This means that this pair is separated in $\mathcal{E}_i$ only possibly by $e,e^*,e'$. This is the pair of edges to lift in this case. It is liftable since it is from two different sides of the complete bipartite part of the lifting graph.

Note that if a pair of edges is not liftable, then it remains so after lifting a liftable pair, so the isolated vertex remains isolated and the complement of the lifting graph remains connected through it. That is, the structure of an isolated vertex plus a balanced complete bipartite graph will be preserved, and $e^*$ remains the isolated vertex throughout the following inductive construction.

Inductively, for the $m$-th time after $i$, we lift a pair of edges $e_{j_m}$ and $e'_{j_m}$ on two different sides of the bipartite part of the lifting graph that are the first edges of the rays $R_m$ and $R_m'$ in $\mathcal{R}$. The infinitely many disjoint paths between $R_m$ and $R_m'$ may go through the edges of the rays $R_{m-1}, \ldots, R_1,R,R^*,R',R_1',\ldots, R'_{m-1}$. These are the rays of $\mathcal{R}$ beginning with the edges $e_{j_{m-1}}, \ldots, e_{j_1}, e,e^*,e',e_{j_1}',\ldots, e'_{j_{m-1}}$ respectively. We continue the lifting process until only $e,e^*,e'$ remain. We join $R_m$ and $R'_m$ by a path $Q_m$ that is edge-disjoint from the rays of $\mathcal{R}$ not yet used, and also is edge-disjoint from all the paths already used in connecting lifted pairs of edges, and from $\{W,W^*,W'\}$. However, this path may go through the edges of the rays $R_{m-1}, \ldots, R_1,R,R^*,R,R_1',\ldots, R'_{m-1}$ at a higher level from $S$ than any one of the previously constructed paths. This finishes the proof of Case 2.
\end{proof}

\section{Immersion}

\begin{par}
Here we find an immersion, in a $2k$-edge-connected locally-finite graph, of a $(2k-1)$-edge-connected finite graph. We only do this for $1$-ended graphs, unlike Theorem 4 of Thomassen in \cite{thomassen2016orientations}, which holds for graphs with multiple ends. We begin by setting out the definition of immersion we are going to use.
\end{par}

\begin{definition}\label{immersion}
For a graph $G$ (finite or infinite), let $\mathcal P(G)$ denote the set of paths in $G$.  An \emph{immersion} of a finite graph $H$ in $G$ consists of an injection $\phi:V(H)\to V(G)$ and a function $\theta:E(H)\to \mathcal P(G)$ such that, for $uv\in E(H)$, 
\begin{itemize}
\item [(1)] $\theta(uv)$ is a $\phi(u)\phi(v)$-path in $G$, 
\item [(2)] for every $v\in V(H)$, the vertex $\phi(v)$ is not an interior vertex of a path in $\theta(E(H))$, and
\item [(3)] the paths in $\theta(E(H))$ are pairwise edge-disjoint.
\end{itemize}
The graph $H$ is \emph{immersed} in $G$ and the subgraph $(\phi(V(H)),X)$ is an \emph{immersion of $H$ in $G$}, where $X$ is the set of all the edges in the paths in $\theta(E(H))$. The vertices of the set $\phi(V(H))$ are called the \textit{branch vertices} of the immersion.
\end{definition}

\begin{theorem}\label{immersion result}
Let $k$ be a positive integer, $G$ a $2k$-edge-connected locally-finite $1$-ended graph, and let $A$ be a finite set of vertices in $G$. Then $G$ contains an immersion $H$ of a finite $(2k-1)$-edge-connected graph $G'$ with vertex set $S\supseteq A$ such that any edge of $G$ with both end-vertices in $A$ is also an edge in $G'$ and in $H$.
\end{theorem}
\begin{proof}

By a standard fact of infinite graphs, which is also a special case of Theorem 1 in \cite{thomassen2016orientations} by Thomassen for multiple ends, there is a finite set $S$ containing $A$ such that each edge of $\delta(S)$ is the first edge of a ray in an edge-disjoint collection of rays $\mathcal{R}$ (all in one, the unique, end of $G$). Let $s$ be the vertex that results from identifying all the vertices of $G-S$. Note that $\deg(s)=|\delta(S)|$. 

Since $2k$ is even, Lemma \ref{lifting sequence} applied with connectivity $2k$ implies there is a sequence of lifts of pairs of edges from $\delta(S)$ such that there is an edge-disjoint linkage of these pairs of edges in $G-S$. The sequence ends by having all the edges of $\delta(S)$ lifted if $|\delta(S)|$ is even or having three of them remaining if it is odd. Since the pairs of edges in the sequence were chosen to be $2k$-liftable, the last graph in the sequence satisfies that for any two vertices $x$ and $y$ in it different from $s$ there are $2k$ edge-disjoint paths between them. Denote this graph by $G^*$. Note that in any case, regardless of whether $|\delta_G(S)|$ is even or odd, $\deg_{G^*}(s)\leq 3$. Thus for any two vertices $x$ and $y$ in $G^*-s$ at most one of the edge-disjoint paths between them in $G^*$ goes through $s$. Thus $G':=G^*-s$ is a $(2k-1)$-edge-connected graph with vertex set $S$. 

Any edge in $G'$ is either an edge of $G$ or an edge that resulted from lifting, that is an edge corresponding to a path in $G$, and this collection of paths is edge-disjoint by construction and the paths have their internal vertices in $G-S$. This gives the desired immersion $H$. Moreover, any edge of $G$ with both end-vertices in $A$ remains in $G^*$, and subsequently $G'$, as the only lifted edges are edges with one end in $S \supseteq A$ and one end in $G-S$. Any edge with both end-vertices in $A$ is also an edge of $H$ as the only edges of $G'$ that are replaced with paths to obtain $H$ are those that resulted from lifting.

\end{proof}

Recall the example in Figure \ref{example 2 lifting from deg 5 to deg 3}, the lifting graph on $5$ edges incident with $s$ is an isolated vertex plus a $K_{2,2}$. There are only $4$ liftable pairs and they are all symmetric. Lifting any one of these liftable pairs gives a lifting graph on $3$ edges incident with $s$ such that no pair of edges is liftable. This means that in this case we cannot continue lifting until we have only one edge incident with $s$. If we could do so, then we would have the immersion of a $2k$-edge-connected graph instead of $2k-1$. 

In other words, connectivity of $(2k-1)$ is the best connectivity one can get for an immersed graph out of the proof approach of successive lifting in the case when the structure of an isolated vertex plus a $K_{2,2}$ shows up. This is also an obstacle to the multiple ends case, since among other issues, one may need to subtract $1$ for different ends which will result in a big reduction in connectivity.

The good news is, we do not need the immersed graph to be $2k$-edge-connected to be able to use it in proving the existence of a $2k$-arc-connected orientation for the infinite graph containing it as an immersion. In \cite{Max2023orientation}, Theorem 3.2, we included the vertex of degree $3$ in the set of branch vertices of the immersion. Having multiple ends meant having more than one such vertex of degree $3$. However, the immersion was constructed carefully such that there is at most only one such vertex in each component of a finite number of components. Still, the immersed graph had the property that between any two vertices in the set $A$ (as in the statement of the theorem above) there are $2k$ edge-disjoint paths.

\section{orientation of finite graphs}

In \cite{thomassen2016orientations} Thomassen presented the following algorithmic way for finding a $k$-arc-connected orientation of a $(4k-2)$-edge-connected finite graph. A path $P$ is \textit{mixed} if each edge in it has a direction, and \textit{directed} if all its edges are directed the same way, so being directed is a special case of being mixed. Similarly, a cycle is directed if all its edges are directed the same way.

 \begin{theorem}\label{algorithm of orientation}\cite{thomassen2016orientations}
 Let $k$ be a positive integer, and let $G$ be a finite $(4k-2)$-edge-connected graph. Successively perform either of the following operations:
 \begin{enumerate}
     \item [O1:] Select a cycle in which no edge has a direction and make it into a directed cycle.
     \item [O2:] Select two vertices $u, v$ joined by $2k-1$ pairwise edge-disjoint mixed paths, and identify $u,v$ into one vertex.
 \end{enumerate}
 When neither of these operations can be performed the resulting oriented graph has only one vertex. The edge-orientations of $G$ obtained by O1 result in a $k$-arc-connected directed graph.
 \end{theorem}

For a detailed proof, we refer the reader to the proofs of Lemma 1 and Theorem 5 in Thomassen's paper \cite{thomassen2016orientations}. However, we give an idea below on how the proof goes so that the subsequent lemma makes sense. 

If we cannot perform operation $O1$, then the set of edges without direction forms a forest, and so there are at most $(n-1)$ of them if $n\geq 2$ is the number of vertices in $G$. 

If we cannot perform operation $O2$, then it can easily be shown that there are at most $(2k-2)(n-1)$ directed edges. Thus if we cannot perform any of the two operations, then $G$ has at most $(2k-1)(n-1)$ edges, contradicting the assumption that it is $(4k-2)$-edge-connected. Therefore, the only case in which none of the two operations can be performed is when the graph consists of one vertex. 

The directed graph obtained at the end of the algorithm can be shown to be $k$-arc-connected inductively as follows. First note that $O1$ gives directions to the edges of any cut in a balanced way. So if a graph can be completely oriented using only operation $O1$, then its resulting arc-connectivity is $(2k-4)/2=2k-1$. This covers the base case. 

If we look at the first application of $O2$, if any, then if $u$ and $v$ are the two vertices to which the operation is applied, there should be $k$ arc-disjoint directed paths from $u$ to $v$ and $k$ arc-disjoint directed paths from $v$ to $u$ as explained below. 

To see why this is the case, note that, since each cut is balanced, then if a cut has at most $(k-1)$ edges in one direction, then it has at most $(k-1)$ edges in the other direction. By Menger's theorem, this means that there are at most $(2k-2)$ mixed paths between $u$ and $v$, and so $O2$ cannot be applied to $u$ and $v$. Now the graph obtained from identifying $u$ and $v$ is a smaller graph and the oriented part of it can be thought of as oriented using only $O1$. Thus by the induction hypothesis the orientation it gets at the end from applying the algorithm to it is $k$-arc-connected.

The following lemma is a straightforward consequence of Theorem \ref{algorithm of orientation}. If a subgraph $H$ of $G$ is oriented using the algorithm in Theorem \ref{algorithm of orientation}, then when the orientation of $H$ is complete, we have a graph $G'$ obtained from $G$ by applying operations $O1$ and $O2$ to $H$, in which $H$ is represented by a vertex. 

Any edge in $G$ not in $H$ whose end-vertices are in $H$ is a loop in $G'$. If $G'$ has only one vertex, then $O1$ can be applied to those loops. If $G'$ has more than one vertex, we continue applying the algorithm until we have only one vertex.

 \begin{lemma}\label{extension of orientation}
 Let $k$ be a positive integer, and let $G$ be a finite $(4k-2)$-edge-connected graph. Let $H$ be a subgraph of $G$ with an orientation obtained using operations O1 and O2. Then the orientation of $H$ can be extended, using O1 and O2, to an orientation of $G$ which is $k$-arc-connected. \qed
 \end{lemma}

 \section{new orientation result}

  Here we present our new orientation result for $1$-ended locally-finite graphs.
 
 \begin{theorem}\label{new orientation result}
 Let $k$ be a positive integer, and let $G$ be a $4k$-edge-connected locally-finite $1$-ended graph. Then $G$ has a $k$-arc-connected orientation.
 \end{theorem}
 \begin{proof}

   The proof is very similar to the proof of Theorem 7 in \cite{thomassen2016orientations}. It differs only in that it does not use Eulerian subgraphs. Since $G$ is locally-finite and connected, it is countable. Let $e_0,e_1,\cdots$ be the edges of $G$. We construct a nested sequence of finite directed subgraphs $\{W_n\}_{n\in \mathbb{N}}$ using operations $O1$ and $O2$ defined in the previous section such that each orientation is an extension of the previous, $W_n$ contains $e_n$, and has the following property: for any two vertices $x$ and $y$ in $V(W_n)$ there are $k$ arc-disjoint directed paths from $x$ to $y$ in $W_{n+1}$. 
   
   The graph $G$ has an edge-connectivity of $4k>1$, therefore it contains a cycle containing $e_0$. Using $O1$, give this cycle an orientation and let $W_0$ be this directed cycle. This defines the first subgraph in the sequence. Note that $W_0$ is not required to be $k$-arc-connected. 
   
   Assume that $W_n$ is defined. To get $W_{n+1}$, let $e_{i_n}$ be the first edge in our enumeration not contained in $W_n$, and let $A$ be the union of $V(W_n)$ and the two end-vertices of $e_{i_n}$. By Theorem \ref{immersion result}, $G$ contains an immersion $H_{n+1}$ of a finite $(4k-1)$-edge-connected graph $G_{n+1}$ such that $A\subseteq V(G_{n+1})$. Note that $V(G_{n+1})\subseteq V(H_{n+1}) \subseteq V(G)$ but $G_{n+1}$ is not necessarily a subgraph of $G$ (but $H_{n+1}$ is). Also, by Theorem \ref{immersion result}, $E(W_n)$ and $e_{i_n}$ are contained in $E(G_{n+1})$, and in $E(H_{n+1})$. In particular, $W_n$ is a subgraph of both $G_{n+1}$ and $H_{n+1}$. 
   
   The graph $W_n$ was oriented using $O1$ and $O2$ and is a subgraph of the $(4k-1)$-edge-connected graph $G_{n+1}$. Thus by Lemma \ref{extension of orientation}, this orientation can be extended, using $O1$ and $O2$, to a $k$-arc-connected orientation of $G_{n+1}$, so $e_{i_n}$ gets oriented as part of this.
   
   An orientation of $H_{n+1}$ can be naturally obtained from an orientation of $G_{n+1}$ by giving each path of the immersion the direction of the edge representing it in $G_{n+1}$. Note that the edge $e_{i_n}$ has the same orientation in $H_{n+1}$ as in $G_{n+1}$. We define $W_{n+1}$ to be the directed graph $H_{n+1}$.
   
   Now it only remains to show that for any two vertices $x$ and $y$ in $V(W_n)$ there are $k$ arc-disjoint directed paths from $x$ to $y$ in $W_{n+1}$. There are such paths from $x$ to $y$ in the oriented $G_{n+1}$. Replacing each edge of these paths with its image under $\theta$ (cf. Definition \ref{immersion}) gives $k$ arc-disjoint directed paths from $x$ to $y$ in $H_{n+1}$, i.e. in $W_{n+1}$ because the paths of $\theta(E(G_{n+1}))$ are edge-disjoint by the definition of immersion \ref{immersion}. 
   
   The union of the directed graphs $W_n$, $n\in \mathbb{N}$, defines an orientation of $G$. For any two vertices $x$ and $y$ of $G$, there exists $n\geq 1$ such that $x$ and $y$ are in $W_n$. To see this consider any path between $x$ and $y$ in $G$. For some sufficiently large $n$, $W_n$ contains all the edges of this path, and so also contains $x$ and $y$. Then there are $k$ arc-disjoint directed paths from $x$ to $y$ in $W_{n+1}$, and so in $G$. Since this is true for every $x$ and $y$ in $G$, the orientation of $G$ is $k$-arc-connected. 
   
   \end{proof}

\bibliographystyle{plain}
\bibliography{reference}
\end{document}